\documentclass[12pt,a4paper]{amsart}

\usepackage{amsmath, amssymb, amsfonts, amsthm, graphicx}
\usepackage[applemac]{inputenc}
\usepackage[all]{xy}
\usepackage{amscd}
\usepackage{enumerate}

\usepackage{a4wide}
\usepackage{comment}

\newtheorem{theorem}{Theorem}[section]
\newtheorem{lemma}[theorem]{Lemma}
\newtheorem{corollary}[theorem]{Corollary}
\newtheorem{proposition}[theorem]{Proposition}

\newtheorem*{question*}{Question}

\theoremstyle{definition}
\newtheorem{definition}[theorem]{Definition}

\theoremstyle{remark}

\newtheorem*{remarks}{Remarks}

\numberwithin{equation}{section}

\newcommand {\N}{\mathbb{N}} %% positive integers
\newcommand {\R}{\mathbb{R}} %% reals
\newcommand{\BB}{\mathcal{B}}

\newcommand{\PP}{\mathcal{P}}

\newcommand{\UU}{\mathcal{U}}
\newcommand{\VV}{\mathcal{V}}

\DeclareMathOperator{\ord}{ord}

\DeclareMathOperator{\INT}{Int}
\DeclareMathOperator{\inte}{Int}

\begin{document}
\title[An analogue of Fekete's lemma]{An analogue of Fekete's lemma for subadditive functions on cancellative amenable semigroups}
\date{\today}
\author[T.Ceccherini-Silberstein]{Tullio Ceccherini-Silberstein}
\address{Dipartimento di Ingegneria, Universit\`a del Sannio, C.so
Garibaldi 107, 82100 Benevento, Italy}
\email{tceccher@mat.uniroma3.it}
\author[M.Coornaert]{Michel Coornaert}
\address{Institut de Recherche Math\'ematique Avanc\'ee,
UMR 7501, Universit\'e  de Strasbourg et CNRS, 7 rue Ren\'e-Descartes, 67000 Strasbourg, France}
\email{coornaert@math.unistra.fr}
\author[F.Krieger]{Fabrice Krieger}
\address{Lyc\'ee G\'en\'eral et technologique Adam de Craponne - 218 rue Chateauredon  13300 Salon-de-Provence, France}
\email{kriegerfabrice@googlemail.com}
\date{\today}
\begin{abstract}
We prove an analogue of Fekete's lemma for subadditive right-subinvariant functions defined on the finite subsets of a cancellative left-amenable semigroup.
This extends results previously obtained in the case of amenable groups by E.~Lindenstrauss and B.~Weiss and by M.~Gromov.
\end{abstract}

\maketitle

% SECTION 1
\section{Introduction}

Fekete's lemma \cite{fekete} is a classical result in undergraduate-level analysis.
 It states that if $(u_n)_{n \geq1}$ is a subadditive sequence of real numbers then 
the sequence $\left(\dfrac{u_n}{n}\right)_{n \geq 1}$ has a limit
$\lambda \in \R \cup \{- \infty\}$ as $n$ tends to infinity (see for example \cite[Proposition
9.6.4]{katok-hasselblatt}).
 The goal of the present paper is to give an analogue of Fekete's lemma for subadditive and 
right-subinvariant functions defined on the set of all finite subsets of a cancellative left-amenable semigroup.
In order to state our main result, let us first recall some basic definitions and 
introduce notation.
\par
Let $S$ be a semigroup, i.e., a set equipped with an associative binary operation.
We denote by $\PP(S)$  the set of all subsets of $S$.
One says that  $S$ is \emph{left-amenable} if
there exists a finitely additive left-invariant probability measure defined on $\PP(S)$, that is, a map
$\mu \colon \PP(S) \to [0,1]$ satisfying the following conditions:
\begin{enumerate}[\rm ({A}1)]
\item 
$\mu(A \cup B) = \mu(A) + \mu(B)$ for all $A,B \in \PP(S)$ such that $A \cap B = \varnothing$;
\item 
$\mu(S) = 1$;
\item 
$\mu(L_s^{-1}(A)) = \mu(A)$ for all $s \in S$ and $A \in \PP(S)$, 
\end{enumerate}
where $L_s \colon S \to S$ denotes the left-multiplication by $s$, that is, the map  defined by 
$L_s(t)=st$ for all $t \in S$.
\par
 One says that $S$ is \emph{right-amenable} if its opposite semigroup is left-amenable.
This is equivalent to the existence of a finitely additive right-invariant probability measure defined on $\PP(S)$, that is, a map
$\mu \colon \PP(S) \to [0,1]$ satisfying (A1), (A2) and\\
\vspace{0.02cm}

%\noindent
%\begin{enumerate}
%\item[\rm (A3')] 
%$\mu(R_s^{-1}(A)) = \mu(A)$ for all $s \in S$ and $A \in \PP(S)$,
%\end{enumerate} 
\noindent
(A3') \ $\mu(R_s^{-1}(A)) = \mu(A)$ for all $s \in S$ and $A \in \PP(S)$,\\
\vspace{0.02cm}

\noindent
where $R_s \colon S \to S$ is the right-multiplication by $s$, that is, the map defined by $R_s(t)=ts$ for all $t \in S$. 
\par
A semigroup is called \emph{amenable} if it is both left-amenable and 
right-amenable.
\par
The notion of an amenable group  was  introduced in 1929  by J.~von~Neumann  
\cite{vN}. 
His original motivation was the study of the Banach-Tarski paradox.
The theory of amenable semigroups was subsequently developed in the 1940s and 1950s 
by M.~Day (see \cite{day-trans}, \cite{day-illinois}, \cite{day-survey}, and the references therein).
Day \cite{day-trans} showed in particular  that every commutative semigroup is amenable, 
thus extending a result previously obtained by von Neumann for groups.
Actually, when passing from groups to semigroups, one encounters  many new phenomena.
For example, left-amenability and right-amenability are equivalent for groups, every finite group is amenable, and every subgroup of an amenable group is itself amenable.
On the other hand, in contrast with the group case, there exist finite semigroups that
are left-amenable but not right-amenable and finite amenable semigroups 
containing semigroups that are neither left-amenable nor right-amenable. 
\par
In the group setting, E.~F\o lner \cite{folner} gave a remarkable combinatorial characterization of amenability by showing that a group $S$ is amenable if and only if it satisfies the following condition:

\begin{enumerate}
\item[{\rm (FC)}] 
for every finite subset $K \subset S$ and every real number $\varepsilon > 0$, there exists a 
non-empty finite subset $F \subset S$ such that  
\begin{equation} 
\label{e:cond-folner-2-ordre-reverser}
\vert kF \setminus F \vert \leq \varepsilon  \vert F \vert \quad \text{for all $k \in K$.} 
\end{equation}
\end{enumerate}
(Here and in the sequel, we use $\vert \cdot \vert$ to denote cardinality of finite sets.)
 Condition (FC)  is known as the \emph{F\o lner condition}.
 \par
In his thesis, A.~Frey \cite{frey} adapted F\o lner arguments to semigroups and showed that every left-amenable semigroup  $S$   satisfies  condition (FC) (see \cite[Theorem 3.5]{namioka} for a simpler proof).
However, (FC) is a necessary but not sufficient condition for left-amenability of semigroups.
Examples of semigroups that  are not left-amenable but satisfy (FC)
are provided by finite semigroups that are not left-amenable 
(observe that  any finite semigroup $S$ trivially satisfies (FC) by taking $F = S$). 
\par
Condition (FC) is equivalent to the 
existence of a directed net $(F_i)_{i \in I}$ of non-empty finite subsets of $S$ such that
\begin{equation}
\label{e:Folner-net-condition}
\lim_{i} \frac{\vert s F_i \setminus F_i \vert}{\vert F_i \vert} = 0 \quad \text{for all $s \in S$.}
\end{equation}
Indeed, if $S$ satisfies (FC), we can construct a directed net $(F_i)_{i \in I}$ 
satisfying \eqref{e:Folner-net-condition} in the following way.
We first take as $I$ the directed set consisting of all pairs $(K,\varepsilon)$, where $K$ is a finite subset of $S$ and  $\varepsilon > 0$, with the partial ordering on $I$ defined by 
$(K_1,\varepsilon_1) \leq (K_2,\varepsilon_2)$  if and only if $K_1 \subset K_2$ and 
$\varepsilon_2 \leq \varepsilon_1$. Then,
for each $i = (K,\varepsilon) \in I$, we take as $F_i$ one of the non-empty finite subsets $F \subset S$ satisfying \eqref{e:cond-folner-2-ordre-reverser}.
Conversely,  suppose that $(F_i)_{i \in I}$ is a directed net of non-empty finite subsets of
$S$ satisfying \eqref{e:Folner-net-condition}.
Let  $K \subset S$ be a finite subset  and $\varepsilon > 0$.
Then, for every $k \in K$, we can find $i_k \in I$ such that 
$|kF_i \setminus F_i| \leq \varepsilon |F_i|$ for all $i \geq i_k$.
To get a non-empty finite subset $F \subset S$ satisfying \eqref{e:cond-folner-2-ordre-reverser},
it suffices to take $F := F_i$, where $i \in I$ is such that $i \geq i_k$ for every $k \in K$ (the existence of such an index $i$ follows from the fact that $I$ is directed and $K$ is finite).
  \par
A directed net $(F_i)_{i \in I}$ of non-empty finite subsets of $S$ satisfying 
\eqref{e:Folner-net-condition}
  is called a \emph{left-F\o lner net} of $S$.
  \par
Recall that an element $s$ in a semigroup $S$ is called \emph{left-cancellable} (resp. \emph{right-cancellable}) 
if the map $L_s$ (resp. $R_s$) is injective.
One says that $s$ is cancellable if it is both left-cancellable and right-cancellable.
The semigroup $S$ is called
\emph{left-cancellative} (resp. \emph{right-cancellative}, resp. \emph{cancellative}) if every element in $S$ is
 left-cancellable  (resp.  right-cancellable, resp. cancellable).
 \par
When $S$ is a  left-cancellative semigroup, it is known that the left-amenability of $S$ is equivalent to the F\o lner condition (FC), and hence to the existence of a left-F\o lner net  (see \cite[Corollary 4.3]{namioka}).
  \par
 The purpose of the present  paper is to establish the following result.

\begin{theorem}
\label{theOWGromov}
Let $S$ be a cancellative left-amenable semigroup
and let $\PP_{fin}(S)$ denote the set of all finite subsets of $S$.
Let  $h\colon \PP_{fin}(S) \to \R$ be a real-valued map satisfying the following conditions:
\begin{enumerate}[\rm (H1)]
\item 
$h$ is subadditive, i.e.,
\[
h(A \cup B)\le h(A)+h(B)  \ \text{ for all $A, B \in \PP_{fin}(S)$};
\]
\item 
$h$ is right-subinvariant, i.e.,
\[
h(As)\leq h(A) \ \text{ for all $s\in S$ and $A \in \PP_{fin}(S)$};
\]
\item 
$h$ is bounded on singletons, i.e., there exists a real number $M \geq 0$ such that
\[
h(\{s\}) \leq M \ \text{ for all $s\in S$}.
\]
 \end{enumerate}
 Then there exists a real number $\lambda  \geq 0$, depending only on $h$,  such that
the net $\left(\dfrac{h(F_i)}{\vert F_i\vert}\right)_{i \in I}$ converges to $\lambda$
for every left-F\o lner net $(F_i)_{i \in I}$ of $S$. 
 \end{theorem}
 
 Observe that conditions (H3) is implied by (H2) when $S$ admits an element $s_0$ such that 
the map  $L_{s_0}$ is onto.
 Indeed, in this case, (H2)   gives us  $h(\{s\}) \leq h(\{s_0\})$ for all $s \in S$.   
 This happens for example when $S$ is a monoid (i.e., $S$ admits an identity element) since we can then take 
 $s_0 = 1_S$.
Thus, an immediate consequence of Theorem \ref{theOWGromov} is the following result.

\begin{corollary}
 \label{cor-theOWGromov}
Let $S$ be a cancellative left-amenable monoid
 and let  $h\colon \PP_{fin}(S) \to \R$ be a subadditive and right-subinvariant map. 
 Then there exists a real number $\lambda  \geq 0$, depending only on $h$,  such that
the net $\left(\dfrac{h(F_i)}{\vert F_i\vert}\right)_{i \in I}$ converges to $\lambda$
for every left F\o lner net $(F_i)_{i \in I}$ of $S$. 
  \end{corollary}

As far as we know, the above result is new even in the case when $S$ is the additive monoid $\N$ of non-negative integers.
However, when $S$ is an amenable group, it was previously established by E.~Lindenstrauss and B.~Weiss \cite[Theorem 6.1]{lindenstrauss-weiss} under the additional assumption that $h$ is 
non-decreasing ($h(A) \leq h(B)$ for all $A,B \in \PP_{fin}(S)$ such that $A \subset B$),
and by M.~Gromov \cite[Section 1.3.1]{gromov}.
\par
Let us note that, in the case when $S$ is a group, condition (H2) implies that $h$ is right-invariant since we then have
$h(A) = h(Ass^{-1}) \leq h(As)$ and hence $h(As) = h(A)$ for all $s \in S$ and $A \in \PP_{fin}(S)$.
This is no more true in general for semigroups.
For example, if $S$ is the additive monoid $\N$,  then the map
$h \colon \PP_{fin}(\N) \to \R$ defined by $h(A) = (1 + \max(A))^{-1} \vert A \vert$ clearly satisfies
(H1) and (H2) but $h(A + s) < h(A)$ for $A = \{0\}$ and $s = 1$.
\par
The proof of Lindenstrauss and Weiss is based on the Ornstein-Weiss machinery of quasi-tiles
(cf.   \cite[Section I.2: Theorem 6]{ornstein-weiss}).
This is the reason why the group version of Theorem \ref{theOWGromov}
is sometimes called the \emph{Ornstein-Weiss lemma} although it does not appear explicitly in \cite{ornstein-weiss}.
A detailed exposition of Gromov's proof of the Ornstein-Weiss lemma may be found in  
\cite{krieger}.
   \par
In the theory of dynamical systems, Theorem \ref{theOWGromov} is important in defining numerical invariants  such as topological entropy, measure-theoretic entropy,
and mean topological dimension. 
These invariants are obtained by taking limits  of quantities defined from a left-F\o lner net 
and one can deduce  from Theorem \ref{theOWGromov} that the choice of the left-F\o lner net is actually irrelevant for actions of cancellative left-amenable semigroups.
\par
Our proof of Theorem \ref{theOWGromov} is entirely self-contained.
 At several points, it is  inspired by some of the ideas developed by Gromov in 
 \cite[Section 1.3.1]{gromov}. However, here again, the passage from groups to semigroups inevitably imposes significant modifications in the arguments.
\par
 We do not know to what extend Theorem \ref{theOWGromov} remains valid for non-cancellative left-amenable semigroups.
\par
 The paper is organized as follows.
In  Section \ref{sec:boundaries-semigroups}, we establish some general properties of boundary sets  in semigroups that are needed for the proof of our main result.
We give in particular a characterization of F\o lner nets  for cancellative semigroups
in terms of relative amenability.
In Section \ref{sec:filling}, we introduce a notion of $\varepsilon$-filling pattern for finite subsets of semigroups and prove a theorem about the existence of certain fillings  by finite systems of tiles with small relative amenability (Theorem \ref{th:quasi-tile}).
This filling theorem is
 a key tool in the proof of Theorem \ref{theOWGromov} given in Section \ref{sec:proof-main-result}. 
 In Section \ref{sec:applications}, which is mostly expository,
 we discuss  the above-mentioned applications of Theorem \ref{theOWGromov} to the definition of numerical invariants of dynamical systems. 

% SECTION 2
  \section{Boundaries and relative amenability}
\label{sec:boundaries-semigroups}

Let $S$ be a semigroup. Let $K$  and $A$ be subsets of $S$.
\par  
The \emph{right $K$-interior} of $A$ is the set 
$$
\inte_K(A) :=\{s\in A : Ks\subset A\}
$$
consisting of all the elements $s$ in $A$ such that
the right-translate of $K$ by $s$  is entirely contained in $A$. 
\par
 The \emph{right $K$-boundary} of $A$ is the set $\partial_K(A) \subset A$ defined by 
$$
\partial_K(A):=A \setminus \inte_K(A).
$$
Thus,   an element $s \in S$ is in  $\partial_K(A)$ if and only if $s$ is in $A$ and $Ks$ 
meets the complement of $A$ in $S$.

  \begin{proposition}
  \label{prodeffrontiere}
Let $S$ be a semigroup. Let $K$, $A$, and $B$ be subsets of $S$. 
Then one has
\begin{enumerate}[\rm(i)]
% (i)
\item
$\partial_{K}(A) = A\cap \left(\bigcup_{k\in K}{L_k}^{-1}(kA \setminus A) \right)$;
% (ii)
\item
if every element in $K$ is left-cancellable then
$\partial_{K}(A) =   \bigcup_{k\in K}{L_k}^{-1}(kA \setminus A)$;
% (iii)
 \item 
$\partial_K(A \cup B) \subset \partial_K(A) \cup \partial_K(B)$;
% (iv)
\item 
$\partial_K(B\setminus A) \subset 
\left(\partial_K(B) \cup\left(\bigcup_{k\in K}L_{k}^{-1}(A\cap kS)\right)\right) \setminus A$;
 % (v)
\item 
if $s$ is a right-cancellable element of $S$ then 
\begin{equation*}
\label{q:R-s-1}
\left(\inte_{K}(A)\right)s = \inte_{K}(As); 
\end{equation*}
% (vi)
\item
if $s$ is a right-cancellable element of $S$ then 
 \begin{equation*}
\label{q:R-s-2}
\left(\partial_{K}(A)\right)s =  \partial_{K}( As). 
\end{equation*}
\end{enumerate}
\end{proposition}

\begin{proof}
(i) This is clear since $s \in \partial_{K}(A)$ if and only if $s \in A$ and $ks \notin A$ for some $k \in K$.
\par
(ii) This immediately follows from (i) since the injectivity of $L_k$ implies that  ${L_k}^{-1}(kA \setminus A) \subset A$. 
\par
(iii) Let $s\in  \partial_K(A \cup B)$. This means that $s\in A \cup B$ and 
$$
Ks \cap (S \setminus (A\cup B))\neq \varnothing.  
$$ 
Since $S \setminus (A\cup B)=(S\setminus A) \cap (S\setminus B)$, we deduce that 
$s\in \partial_K(A) \cup \partial_K(B)$. 
\par
(iv) Suppose that  $s\in \partial_K(B\setminus A)$. This means that  $s\in B\setminus A$ and 
$$
Ks \cap (S \setminus (B \setminus A)) \neq \varnothing.
$$
Since $S \setminus (B\setminus A)= (S \setminus B) \cup  A$, we deduce that if 
$s\notin \partial_K(B)$, then 
$Ks \cap A \neq \varnothing$ and hence  $s\in \bigcup_{k\in K} L_{k}^{-1}(A \cap kS)$. 
As $s \notin A$, inclusion (iv) immediately  follows.
\par
 (v) 
Suppose that $s \in S$ is right-cancellable and let $g\in \left(\inte_{K}(A)\right)s$. 
This means that there exists $a\in \inte_K(A)$ such that $g=as$. Hence $g\in As$ and $Kg= K(as)=(Ka)s \subset As $ since $a \in \inte_K(A)$. Thus $g\in \inte_K(As)$. This gives the inclusion 
$\left(\inte_{K}(A)\right)s\subset \inte_K(As) $.
\par
Conversely,  suppose now that  $g\in \inte_K(As)$. Then $g\in As$ and $Kg \subset As$. Thus, there exists $a\in A$ such that $g=as$ and $(Ka)s=K(as) \subset As$. Remark that the inclusion $(Ka)s \subset As$ is equivalent to the inclusion $Ka\subset A$ by injectivity of $R_s$. This proves that $a\in \inte_K(A)$ so that $g\in \left(\inte_{K}(A)\right)s$. Hence $\inte_K(As)  \subset \left(\inte_{K}(A)\right)s$. This completes the proof of (v).
\par	
(vi)
If $s \in S$ is right-cancellable, we have
\begin{align*}
\left(\partial_{K}(A)\right)s  &= \left(A\setminus \inte_K(A)\right) s  \\
&= As \setminus \left(\inte_{K}(A)\right)s  && \text{(since $R_s$ is injective)} \\
&= As \setminus \inte_K(As) && \text{(by (v))} \\
&= \partial_{K}(As).
\end{align*}
 This shows (vi).
\end{proof}

 \begin{lemma}\label{boundary-properties}
Let $S$ be a  semigroup. 
Suppose that $K$ and $A$ are finite subsets of $S$ and that every element of $K$ is left-cancellable.
Then one has 
\begin{equation}
\label{in:boundary}
|\partial_{K}(A)|\leq \sum_{k \in K} \vert kA \setminus A \vert
\end{equation} 
and
\begin{equation}
\label{in:boundary-2}
|kA \setminus A| \le |\partial_{K}(A)| \quad \text{ for all } k\in K.
\end{equation} 
\end{lemma}

\begin{proof}
It follows from  Proposition \ref{prodeffrontiere}.(ii) that
\begin{equation}
\label{e:formula-for-boundary}
\partial_{K}(A) =   \bigcup_{k\in K}{L_k}^{-1}(kA \setminus A).
\end{equation}
This implies
$$
\vert \partial_{K}(A) \vert = \vert \bigcup_{k\in K}{L_k}^{-1}(kA \setminus A) \vert \leq
\sum_{k\in K} \vert {L_k}^{-1}(kA \setminus A) \vert.
$$
As $\vert {L_k}^{-1}(kA \setminus A) \vert = \vert  kA \setminus A \vert$ for all $k \in K$ by
injectivity of $L_k$, this gives us \eqref{in:boundary}.
\par
On the other hand, given $k \in K$,
we deduce from \eqref{e:formula-for-boundary} that
$$
{L_k}^{-1}(kA \setminus A)\subset  \partial_{K}(A).
$$
This implies
$$
\vert kA \setminus A \vert = \vert{L_k}^{-1}(kA \setminus A) \vert \leq \vert \partial_{K}(A) \vert
$$
which yields \eqref{in:boundary-2}.
\end{proof}

Let  $A$  and $K$ be subsets of $S$ with  $A$ finite and non-empty. 
Then $\partial_K(A)$ is also finite since $\partial_K(A) \subset A$.  
We define the \emph{amenability constant} 
of $A$ with respect to $K$ by
$$
\alpha(A,K) := \frac{\vert \partial_K(A)\vert}{\vert A\vert}.
$$
Note that $\alpha(A,K)$ is rational and that one has $0 \leq \alpha(A,K) \leq 1$.
\par
For left-cancellative semigroups, left-amenability is equivalent to the existence of  finite subsets with arbitrary small relative amenability.
More precisely, we have the following result.

\begin{proposition}
\label{p:amenability-boundary}
Let $S$ be a left-cancellative semigroup. 
Then the following conditions are equivalent:
\begin{enumerate}[\rm (a)]
\item 
$S$ is left-amenable;
\item  
for every finite subset $K$ of $S$ and every real number  $\varepsilon>0$, 
there exists a non-empty finite subset 
$F$ of $S$ such that  $\alpha(F,K)\leq \varepsilon$.
 \end{enumerate}
 \end{proposition}
 
 \begin{proof}
 Let $F$ and $K$ be finite subsets of $S$ with $F \not= \varnothing$.
From inequality \eqref{in:boundary} of
  Lemma \ref{boundary-properties}, we deduce that if 
  $|kF \setminus F| \leq \varepsilon \vert F \vert$  for all $k \in K$, then
  $\alpha(F,K) \leq  \vert K \vert \varepsilon$.
  Conversely, inequality \eqref{in:boundary-2}  implies that if
  $\alpha(F,K) \leq \varepsilon$ then
  $\vert kF \setminus F \vert \leq \varepsilon \vert F \vert $ for all $k \in K$.
  \par
We deduce that the F\o lner condition (FC) is equivalent to condition (b) of the statement.
On the other hand, as $S$ is left-cancellative, we know from the result mentioned in the Introduction that $S$ is left-amenable if and only if it satisfies (FC).
This shows the equivalence between conditions (a) and (b).
  \end{proof}

Similarly, we have the following characterization of F\o lner nets in left-cancellative and left-amenable semigroups. 

 \begin{proposition}
 \label{p:Folner-boundaries}
Let $S$ be a left-cancellative and left-amenable semigroup.
Let $(F_i)_{i \in I}$ be a directed net of non-empty finite subsets of $S$.
Then the following conditions are equivalent:
 \begin{enumerate}[\rm (a)]
 \item
 $(F_i)_{i \in I}$ is a left-F\o lner net for $S$;
 \item
 for each finite subset $K$ of $S$, one has $\lim_i \alpha(F_i,K) = 0$.
 \end{enumerate}  
\end{proposition}

\begin{proof}
 Let $s \in S$ and take $K = \{s\}$.
Then one has $|sF_i \setminus F_i|/|F_i| \leq \alpha(F_i,K)$ for all $i \in I$ by 
\eqref{in:boundary-2}.
This shows that (b) implies (a).
\par
Conversely, suppose that $(F_i)_{i \in I}$ is a left-F\o lner net for $S$. 
Let $K$ be a finite subset of $S$
and $\varepsilon > 0$.
Then  there exists $i_k \in I$ such that
$|kF_i \setminus F_i|/|F_i| \leq \varepsilon$ for all $i \geq i_k$.
If $j \in I$ is such that $j \geq i_k$ for all $k \in K$, we deduce 
that $\alpha(F_i,K) \leq \varepsilon |K|$ for all $i \geq j$ by using \eqref{in:boundary}. This shows that (a) implies (b).
    \end{proof}

% SECTION 3
\section{Fillings}
\label{sec:filling}

 The goal of this section is to establish Theorem \ref{th:quasi-tile}
which is a key tool in the proof of Theorem \ref{theOWGromov} that will be given in the next section.

\begin{definition}
Let $X$ be a set and $\varepsilon>0$ a real number. A family
$(A_j)_{j\in J}$ of finite subsets of $X$ is said to be
\emph{$\varepsilon$-disjoint} if there exists a family $(B_j)_{j \in J}$ of pairwise disjoint subsets of $X$ such that $B_j \subset A_j$ and
$|B_j|\ge (1-\varepsilon)|A_j|$ for all $j\in J$.
\end{definition}

\begin{lemma}
\label{lemepsilondisjoint}
Let $X$ be a set and $(A_j)_{j\in J}$ a finite
$\varepsilon$-disjoint family of finite subsets of $X$. Then one has
$$
(1-\varepsilon)\sum_{j\in J} |A_j| \le 
\left|\bigcup_{j\in J} A_j\right|.$$
\end{lemma}

\begin{proof}
Since $(A_j)_{j\in J}$ is $\varepsilon$-disjoint, there exists
a  family $(B_j)_{j\in J}$ of pairwise disjoint subsets of $X$ such that $B_j\subset A_j$ and $|B_j| \ge (1-\varepsilon)\ |A_j|$ for all $j\in J$. Thus, we have
$$
(1-\varepsilon)\sum_{j\in J} |A_j| \le \sum_{j\in J} |B_j| = \left| \bigcup_{j\in J}
B_j \right|\le
  \left| \bigcup_{j\in J} A_j \right|.
  $$
 \end{proof}

\begin{lemma} 
\label{lemUnion}
Let $S$ be a semigroup. Let also $K$ be a finite subset of $S$ and $0< \varepsilon<1$.
Suppose that $(A_j)_{j\in J}$ is a finite $\varepsilon$-disjoint family of non-empty finite subsets of $S$. Then one has
$$
\alpha\left(\bigcup_{j\in J} A_j,K\right)\le \frac{1}{1-\varepsilon}\max_{j\in J}\alpha(A_j,K).
$$
\end{lemma}

\begin{proof} 
Let us set $M:=\max_{j\in J}\alpha(A_j,K)$.
It follows from Proposition \ref{prodeffrontiere}.(iii) that
$$
\partial_{K}\left( \bigcup_{j\in J} A_j \right) \subset \bigcup_{j\in J} \partial_{K} (A_j).
$$ 
Thus
\begin{equation*}
\left|\partial_K\left(\bigcup_{j\in J} A_j\right)\right| \le \left|\bigcup_{j\in J}
 \partial_{K} (A_j)\right|
 \le \sum_{j\in J} |\partial_{K}
 (A_j)|=\sum_{j\in J} \alpha(A_j,K)
 |A_j|\le M \sum_{j\in J}
 |A_j| .
 \end{equation*}
 As the family $(A_j)_{j\in J}$ is $\varepsilon$-disjoint, we deduce from 
 Lemma \ref{lemepsilondisjoint} that
 $$
\alpha\left(\bigcup_{j\in J} A_j,K\right)=\frac{\left|\partial_{K}\left(\bigcup_{j\in J} A_j\right)\right|}{\left|\bigcup_{j\in J} A_j\right|} \le \frac{M}{1-\varepsilon}.  %\qed
$$
\end{proof}

\begin{lemma}
\label{lemAsetminusB}
Let $S$ be a semigroup. Let $K$, $A$ and $\Omega$ be finite subsets of $S$ such that
every element of $K$ is left-cancellable and  
$\varnothing \neq A \subset \Omega$. 
Suppose that  $\varepsilon > 0$ is a real number such that
$|\Omega \setminus A| \ge \varepsilon |\Omega|$.
Then one has
$$
\alpha(\Omega \setminus A,K) \le \frac{\alpha(\Omega,K)+ |K| \alpha(A,K)}{\varepsilon}.
$$
\end{lemma}
 
\begin{proof}
By Proposition \ref{prodeffrontiere}.(iv), we have that
\begin{equation*}
\partial_{K}(\Omega\setminus A) \subset \left(\partial_{K}(\Omega) \cup \left(\bigcup_{k\in K} L_{k}^{-1}(A\cap kS)\right)\right) \setminus A.
 \end{equation*}
This implies
\begin{equation*}
\partial_{K}(\Omega\setminus A) \subset \partial_{K}(\Omega) \cup \left(\bigcup_{k\in K} L_{k}^{-1}(A\cap kS) \setminus A\right) 
\end{equation*}
and hence
\begin{equation} 
\label{eqn-1AminusB}
|\partial_{K}(\Omega\setminus A)|  
\le   |\partial_{K}(\Omega)| + \sum_{k \in K}|L_{k}^{-1}(A\cap kS) \setminus A|.
\end{equation}
Now, for all $k \in K$, the injectivity of $L_k$ implies that 
$$
|L_{k}^{-1}(A\cap kS) \setminus A|=|k(L_{k}^{-1}(A\cap kS) \setminus A)|=|A\cap kS \setminus kA|\le
|A \setminus kA|=|kA \setminus A|,
$$
where the last equality follows from the fact that $A$ and $kA$ have the same cardinality.
Hence, by using inequality \eqref{in:boundary-2} in Lemma \ref{boundary-properties},   
we get 
\begin{equation}
\label{eqn-1AminusB-2}
|L_{k}^{-1}(A\cap kS) \setminus A|\le \vert \partial_K(A) \vert  
\end{equation}
for all $k \in K$.
\par
From \eqref{eqn-1AminusB} and \eqref{eqn-1AminusB-2}, we deduce that
\begin{equation} 
\label{eqn-2AminusB}
|\partial_{K}(\Omega\setminus A )| \le |\partial_{K}(\Omega)| + |K| \vert \partial_K(A) \vert .
\end{equation}
It follows that
\begin{align*}
\alpha(\Omega \setminus A,K) &=\frac{|\partial_{K}(\Omega\setminus A )|}{|\Omega \setminus A|} \\
&\le  \frac{|\partial_{K}(\Omega)| + |K| \vert \partial_K(A) \vert }{|\Omega \setminus A|} 
&& \text{(by \eqref{eqn-2AminusB})} \\
&= \frac{\alpha(\Omega,K)|\Omega| + |K| \alpha(A,K)|A| }{|\Omega \setminus A|} \\
&\le \frac{\alpha(\Omega,K)|\Omega| + |K| \alpha(A,K) |A| }{\varepsilon |\Omega|} 
&& \text{(since $|\Omega \setminus A| \ge \varepsilon |\Omega|$)} \\  
&\le
 \frac{\alpha(\Omega,K)+ |K| \alpha(A,K)}{\varepsilon}
 && \text{(since $|A| \le |\Omega|$).}
 \end{align*}
\end{proof}

\begin{lemma}
\label{lemcardinalAB}
 Let $S$ be a   semigroup.
Let $A$ and $B$ be finite subsets of $S$.
 Suppose that every element of $A$ is left-cancellable. 
 Then one has
$$
\sum_{s\in S} |As\cap B|\le |A| |B|.
$$
 \end{lemma}

 \begin{proof}
For $E \subset S$, denote by
 $\chi_E \colon S \to \R$  the characteristic map of $E$, i.e., the map defined by $\chi_E(s) = 1$ if $s \in E$ and $\chi_E(s) = 0$ otherwise.
We have
\begin{equation}
\label{eq:AsB} 
\begin{split}  \sum_{s\in S} |As \cap B| &= \sum_{s\in S}
 \sum_{s' \in
 S}\chi_{As\cap B}(s')\\
&=\sum_{s \in S}
 \sum_{s' \in
 S}\chi_{As}(s')\chi_B(s')\\
&= \sum_{s' \in S}
 \sum_{s \in
 S}\chi_{As}(s')\chi_B(s') \\
&= \sum_{s' \in S} \chi_B(s') \left(\sum_{s \in S}\chi_{As}(s')\right). 
  \end{split}
\end{equation}
If we fix $a \in A$ and $s' \in S$,
the injectivity of  $L_a$ implies that   there exists at most one element $s \in S$ such that $as = s'$.
It follows that, given $s' \in S$,  there are at most $\vert A \vert$ elements $s \in S$ such that $s' \in As$. In other words, we have
$$\sum_{s \in S}\chi_{As}(s')\le  |A|$$
for all $s'\in S$.
Thus, we deduce from   \eqref{eq:AsB} that 
 $$
 \sum_{s\in S} |As \cap B| \le |A|\sum_{s' \in S}
\chi_ B(s')=|A||B|.  
% \qed
$$
\end{proof}

\begin{remarks}
1) The argument used in the preceding proof shows that the inequality in  
Lemma \ref{lemcardinalAB}
can be replaced by an equality if $L_a$ is bijective for every $a \in A$
(e.g., if $S$ is a group).
In fact, when $S$ is a group, the equality
$\sum_{s\in S} |As\cap B|= |A| |B|$ 
is obtained by taking $f = \chi_A$ and $g = \chi_B$ in the  formula  
$\Vert f \ast g \Vert_1 = \Vert f \Vert_1 \Vert g \Vert_1$, valid for all $f,g \geq 0$ in the convolution Banach algebra $\ell^1(S)$. 
 \par
2) The inequality in Lemma \ref{lemcardinalAB} may be strict.
Consider for example the additive monoid $\N$ of non-negative integers and two non-empty finite subsets $A,B \subset \N$ with $\max B < \min A$.
Then one has
$\sum_{s \in \N} \vert (A + s) \cap B \vert = 0$ but $\vert A \vert \vert B \vert \geq 1$.
Note that the additive monoid $\N$ is commutative (and hence amenable) and cancellative.
\par
3) Lemma \ref{lemcardinalAB} becomes false if we drop the hypothesis that every element of $A$  
is left-cancellable.
Indeed, consider the monoid $S = \{s_0,s_1\}$, where $s_0$ is an identity element and $s_1 \not=
s_0$ satisfies
$s_1^2 = s_1$. 
Then, by taking  $A = B = \{s_1\}$, we have
$\sum_{s \in S} |As\cap B| = 2$ but $|A| |B| = 1$.  
\end{remarks}

\begin{definition}
Let $S$ be a semigroup. Let $K$ and $\Omega$ be finite subsets of $S$.
Given a real number  $\varepsilon >0$, 
a finite subset $P\subset S$ is called an
\emph{$(\varepsilon, K)$-filling pattern} for $\Omega$ if the following conditions are satisfied:
\begin{enumerate} [(F1)]
\item 
$P\subset \inte_K(\Omega)$;
\item 
the family $(Ks)_{s\in P}$ is $\varepsilon$-disjoint.
\end{enumerate}
\end{definition}

 % ICI ON VA Y METTRE UNE BELLE FIGURE.
\par
The following lemma will be used in
the proof of  Theorem \ref{th:quasi-tile}
(compare with  $(+)_{\varepsilon}$ in Section 1.3.1 of \cite{gromov} in the group case).
It can be viewed as a kind of analogue of Euclidean division for integers.

\begin{lemma}[Filling lemma]
\label{lempreliminaire}
Let $S$ be a cancellative semigroup. 
Let $\Omega$ and $K$ be non-empty finite subsets of $S$. 
Then, for every $\varepsilon \in (0,1]$, there exists an 
$(\varepsilon,K)$-filling pattern $P$ for $\Omega$ such that
 \begin{equation}
 \label{e;cond-ii-filling}
  |KP| \ge 
 \varepsilon (1-\alpha(\Omega,K))|\Omega|.
 \end{equation} 
\end{lemma}

\begin{proof}
 Let $\PP$ denote the set consisting of all $(\varepsilon,K)$-filling patterns for $\Omega$. 
Observe that $\PP$ is not empty, since $\varnothing \in \PP$, and that every element of $\PP$
has cardinality bounded above by $|\inte_K(\Omega)|$,
since it is contained in $\inte_K(\Omega)$.
 Choose a pattern $P \in \PP$ with maximal cardinality.
 Let us show that \eqref{e;cond-ii-filling} is satisfied.
 To slightly simplify notation, let us set 
 $$
 B :=KP = \bigcup_{s\in P}Ks.
 $$ 
 By applying Lemma \ref{lemcardinalAB}, we get
\begin{eqnarray}
\label{inegalite-1}
\sum_{s\in \INT_K(\Omega)} |Ks \cap B| \le \sum_{s\in S} |Ks \cap B| \le |K||B|.
\end{eqnarray}
Let us prove that
 \begin{equation}
 \label{eqnmaximalite}
 \varepsilon |Ks| \le |Ks \cap B| \quad  \text{for all  } s\in \INT_K(\Omega).
 \end{equation}
If $s \in P$, then $Ks \cap B=Ks$ and
 \eqref{eqnmaximalite} holds true since $\varepsilon \le 1$. 
Let now $s\in
 \INT_K(\Omega) \setminus P$ and suppose, by contradiction, that 
 $|Ks \cap B| < \varepsilon|Ks|$. Then, we have that
 $$
 |Ks \setminus B |=|Ks|-|Ks\cap B|>|Ks|-\varepsilon |Ks|=
 (1-\varepsilon)|Ks|,
 $$
which implies that $P\cup \{s\}$ is an
 $(\varepsilon,K)$-filling pattern for $\Omega$. This contradicts
 the maximality of the cardinality of $P$. This proves \eqref{eqnmaximalite}. 
 \par
Finally, we obtain
\begin{align*}
 \varepsilon|K||\INT_K(\Omega)|
  &= \sum_{s \in \INT_K(\Omega)} \varepsilon |K| \\
  &= \sum_{s \in \INT_K(\Omega)} \varepsilon |Ks| && \text{(since $|K| = |Ks|$ by right-cancellativity of $s$)} \\
 &\le \sum_{s\in \INT_K(\Omega)}|Ks \cap B| && \text{(by \eqref{eqnmaximalite})} \\ 
& \leq |K| |B| && \text{(by \eqref{inegalite-1}),} 
  \end{align*}
which gives us
 $$
 |B|\ge \varepsilon |\inte_K(\Omega)|.       
 $$
 As $|\inte_K(\Omega)| = |\Omega| - |\partial_{K} (\Omega)| = (1-\alpha(\Omega,K))|\Omega|$,
this yields \eqref{e;cond-ii-filling}.
 \end{proof}

 \begin{theorem}[Filling theorem] 
 \label{th:quasi-tile}
Let $S$ be a cancellative  semigroup and 
let $\varepsilon \in (0,\dfrac{1}{2}]$.
Then there exists an integer $n_0= n_0(\varepsilon) \ge 1$ such that for each integer $n\ge n_0$ the following holds. 
\par
If $(K_j)_{1 \leq j \leq n}$ is a finite sequence of non-empty finite subsets of $S$ such that
\begin{equation}
\label{e:conditions-Kj}
\alpha(K_k, K_j)\le \frac{\varepsilon^{2n}}{|K_j|} \quad \text{for all $1\le j<k \le n$},
\end{equation}
and $D$ is a non-empty finite subset of $S$ such that
\begin{equation}
\label{e:condition-D}
\alpha(D,K_j)\le \varepsilon^{2n} \quad \text{for all  } 1 \leq j \leq n,
\end{equation}
then there exists a finite sequence $(P_j)_{1 \leq j \leq n}$ of finite subsets of $S$ satisfying the following conditions:
\begin{enumerate}[\rm (T1)]
\item
 the set $P_j$ is an $(\varepsilon,K_j)$-filling pattern of $D$ for every $1 \leq j \leq n$;
\item
the subsets $K_jP_j \subset D$, $1 \leq j \leq n$,  are pairwise disjoint;
\item
the subset $D' \subset D$ defined by
$$
D' := D \setminus \bigcup_{1 \leq j \leq n} K_j P_j
$$
is such that  $|D'| \leq \varepsilon |D|$.
\end{enumerate} 
 \end{theorem}

% ICI ON VA Y METTRE UNE AUTRE BELLE FIGURE.

\begin{proof}
Fix $\varepsilon \in (0,\dfrac{1}{2}]$ and a positive integer $n$.
 Let $K_j$, $1 \leq j \leq n$, and $D$ be non-empty finite subsets of $S$
 satisfying conditions \eqref{e:conditions-Kj} and \eqref{e:condition-D}.
 \par 
Let us  first define, by induction, a finite process with at most $n$ steps for constructing suitable finite subsets  $P_n,P_{n - 1}, \dots,P_1$  of $S$.
We will see that these subsets have the required properties when $n$ is large enough, i.e., for
$n \geq n_0$ with $n_0 = n_0(\varepsilon)$ that will be made precise at the end of the proof.

 \noindent
\textbf{Step 1.} 
We set $D_0 := D$.
By \eqref{e:condition-D}, we have

 \begin{enumerate}
 \item[{\rm (H(1;a))}] \
 $\alpha(D_{0}, K_j) \le \varepsilon^{2n}$ for all $1\le j \le n $.  
 \end{enumerate}

Using Lemma \ref{lempreliminaire} with $\Omega=D_0 =D$ and $K = K_n$, 
we can find a finite subset
$P_n \subset S$ such that
 
\begin{enumerate}
 \item[{\rm (H(1;b))}]  \
 $P_{n}$ is an $(\varepsilon, K_{n})$-filling pattern for $D_{0}$
 \end{enumerate}

and 
\begin{equation}
\label{eq:H-2-2}
|K_n P_n| \ge 
\varepsilon \big(1-\alpha(D,K_n)\big)|D| \ge \varepsilon (1- \varepsilon^{2n})|D|.
\end{equation}

\begin{enumerate}
\item[{\rm (H(1;c))}] \
Setting
$$
D_1:=D_0 \setminus  K_nP_n,
$$
we deduce  from \eqref{eq:H-2-2} that
\begin{equation*}\label{eqn-inegaliteD1}
|D_1|\le |D| \big(1-\varepsilon (1-\varepsilon^{2n})\big).
\end{equation*}
 \end{enumerate}

\noindent
\textbf{Step $\boldsymbol{k}$.} 
We continue this  process by induction as follows. Suppose that the process has been applied $k$ times, with $1 \le k \le n-1$.
It is assumed that the induction hypotheses at step $k$ are the following:
 \begin{enumerate}
 \item[{\rm (H(k;a))}] \
 $D_{k - 1}$ is a subset of $D$ satisfying 
 $$
\alpha(D_{k-1}, K_j) \le (2k - 1) \varepsilon^{2n-k+1} \quad \text{for all  }1\le j \le n-k+1;
 $$
 \item[{\rm (H(k;b))}] \ 
$P_{n-k+1}\subset S$ is an $(\varepsilon, K_{n-k+1})$-filling pattern for $D_{k-1}$;
 \item[{\rm (H(k;c))}] \
 setting 
 $$
 D_k:=D_{k-1} \setminus K_{n-k+1}P_{n - k + 1},
 $$ 
 we have 
 $$
 |D_k|\le  |D|\prod_{0 \leq i \leq k-1}\left(1-\varepsilon \big(1-(2i+1)\varepsilon^{2n-i}\big)\right).
 $$
 \end{enumerate}
Note that these induction hypotheses are satisfied for $k=1$ by Step 1.
\par
Let us pass to Step $k+1$. \\

\noindent
{\bf{Step $\boldsymbol{k+1}$.}} 
If $|D_k| \le \varepsilon |D_{k-1}|$ and hence $|D_k| \le \varepsilon |D|$,
then we take $P_j = \varnothing$ for all $1 \leq j \leq n - k$ and stop the process.
\par 
Otherwise, we have 
$|D_k|> \varepsilon |D_{k-1}|$. 
 Let us estimate from above, for all $1 \leq j \leq n - k$, the relative amenability constants
$\alpha(D_k,K_j)$.  
 \par
 Let $1 \leq j \leq n - k$.
 \par
 If $P_{n - k + 1} = \varnothing$, then $D_k = D_{k - 1}$ and therefore
\begin{align*}
\alpha(D_k,K_j) 
&= \alpha(D_{k - 1},K_j)\\
 &\leq (2k - 1) \varepsilon^{2n-k+1}  && \text{(by our induction hypothesis (H(k;a)))}\\ 
&\leq (2k + 1) \varepsilon^{2n-k} && \text{(since $0 < \varepsilon < 1$)}.
\end{align*}
Suppose now that $P_{n - k + 1} \not= \varnothing$.
Then we can apply Lemma \ref{lemAsetminusB} with $\Omega := D_{k-1}$ 
and $A := K_{n-k+1}P_{n - k + 1}$. 
This gives us
\begin{equation}
\label{eqn-etapek-alpha}
 \alpha(D_k,K_{j})= \alpha(D_{k-1} \setminus K_{n-k+1}P_{n - k + 1} , K_j) 
 \le \frac{\alpha(D_{k-1},K_{j}) + |K_{j}|\alpha(K_{n-k+1}P_{n-k+1},K_{j}) }{\varepsilon}.
\end{equation}
 Proposition \ref{prodeffrontiere}.(vi) and condition
\eqref{e:conditions-Kj}  imply that, for all $s \in S$,
$$
\alpha(K_{n-k+1}s,K_j)=\alpha(K_{n-k+1},K_j)\le  \frac{\varepsilon^{2n}}{|K_j|}.
$$ 
Since the family $(K_{n-k+1}s)_{s \in P_{n-k+1}}$ is $\varepsilon$-disjoint, 
the preceding inequality together with Lemma \ref{lemUnion} give us
 $$
 \alpha(K_{n-k+1}P_{n - k + 1},K_j) 
 =  \alpha\left(\bigcup_{s\in P_{n-k+1}}K_{n-k+1}s,K_j\right)
 \le  \frac{\varepsilon^{2n}}{(1-\varepsilon)|K_j|}.
 $$
 From inequality
\eqref{eqn-etapek-alpha} and the induction hypothesis (H(k;a)), we deduce that
$$
\alpha(D_{k},K_{j}) \le \frac{(2k-1) \varepsilon^{2n-k+1}}{\varepsilon} +
\frac{\varepsilon^{2n}}{(1-\varepsilon)\ \varepsilon} 
\le (2k+1)\varepsilon^{2n-k}
$$ 
 (for the second inequality, observe that $1/(1-\varepsilon)\le 2$ since $0<\varepsilon \le 1/2$).
 \par 
This shows (H(k+1;a)).
\par 
Using Lemma \ref{lempreliminaire} with $\Omega:=D_k$ and $K:=K_{n-k}$, we can find a finite subset $P_{n-k} \subset S$ such that
$P_{n-k}$ is an $(\varepsilon, K_{n-k})$-filling pattern for $D_k$, thus yielding (H(k+1;b)), and satisfying
\begin{equation}
\label{eq:D-k}
|  K_{n-k} P_{n - k}|\ge \varepsilon\
 \big(1- \alpha(D_k, K_{n-k})\big)|D_k|\ge \varepsilon\ \big(1-(2k+1)\varepsilon^{2n-k} \big)|D_k|.
\end{equation}

Setting 
$$
D_{k+1}:=D_{k} \setminus    K_{n-k} P_{n - k},
$$ 
we deduce from \eqref{eq:D-k} that
$$
|D_{k+1}|\le |D_k|\big(1-\varepsilon \big(1-(2k+1)\varepsilon^{2n-k}\big)\big).
$$
Together with the inequality of the induction hypothesis (H(k;c)), this yields
$$
|D_{k+1}|\le |D|\prod_{0 \leq i \leq k}\left(1-\varepsilon \big(1-(2i+1)\varepsilon^{2n-i}\big)\right).
$$
Thus condition (H(k+1;c)) is also satisfied. This finishes the construction of Step $k+1$ and proves the induction step.

Now, suppose that this process continues until Step $n$.
Using (H(k;c)) for $k=n$, we obtain
\begin{equation}\label{eqn-Dn}
|D_n|\le  |D|\prod_{0 \leq i\leq n-1}\left(1-\varepsilon
\big(1-(2i+1)\varepsilon^{2n-i}\big)\right).
\end{equation}
We will show that for $n \geq n_0$, with $n_0 = n_0(\varepsilon)$ only depending on
$\varepsilon$, we get $|D_n|\le \varepsilon |D|$.
\par
As $(2i+1)\varepsilon^{2n-i} \leq (2n+1)\varepsilon^{n+1}$ for all $0 \leq i \leq n -1$, 
we deduce from  \eqref{eqn-Dn},  that
 \begin{equation}
  \label{eqnfinales}
   |D_n|\le |D| \big(1-\varepsilon(1- (2n+1)\varepsilon^{n+1})\big)^{n}.
\end{equation}
 Since $\lim_{r\to +\infty}(2r+1)\varepsilon^{r+1}=0$ and $\lim_{r\to
 +\infty} (1-\frac{\varepsilon}{2})^{r}=0$, both monotonically for large $r$, we can find an integer $n_0  = n_0(\varepsilon) \geq 1$
such that for all $r\ge n_0$, we have both $(2r+1)\varepsilon^{r+1}\le \frac{1}{2}$ and $(1-\frac{\varepsilon}{2})^{r} \le \varepsilon$. 
Now, if $n\ge n_0$, using inequality \eqref{eqnfinales} we deduce
$$
|D_n|\le |D|\left(1-\frac{\varepsilon}{2}\right)^{n}\le \varepsilon|D|.
$$
  This finishes the proof of the theorem. 
\end{proof}

% SECTION 4
\section{Proof of the main result}
\label{sec:proof-main-result}
  In this section, we give the proof of Theorem \ref{theOWGromov}.
\par
So let $S$ be a cancellative left-amenable semigroup and let $h \colon \PP_{fin}(S) \to \R$ be a real-valued map satisfying conditions (H1), (H2) and (H3).
\par
First observe  that by taking $A=B$ in condition (H1),  we get $h(A)\le 2h(A)$ and hence 
\begin{equation}
\label{e:h-non-negative} 
h(A) \geq 0 \quad \text{for all  }A\in \PP_{fin}(S).
\end{equation} 
On the other hand, we deduce from (H1)  that
$$
h(A) = h\left(\bigcup_{s \in A} \{s\}\right) \leq \sum_{s \in A} h(\{s\}) 
$$
so that, by using (H3), we get
\begin{equation}
\label{e:h-bound} 
h(A)  \leq M|A| \quad \text{for all} A\in \PP_{fin}(S).
\end{equation}

Let $(F_i)_{i \in I}$ be a left-F\o lner net for $S$.
By Proposition \ref{p:Folner-boundaries}, we have 
\begin{equation}
\label{e:boundary-Folner}
\lim_i \alpha(F_i,K) = 0 \quad \text{for every finite subset $K \subset S$}.
\end{equation}
 
 Consider the quantity
\begin{equation}
\label{eq:lambda}
\lambda := \liminf_i \frac{h(F_i)}{|F_i|}.
\end{equation}
Note that $0 \leq \lambda \leq M$ by \eqref{e:h-non-negative} and \eqref{e:h-bound}.
 \par
  Recall that one says that a finite sequence $(K_j)_{1 \leq j \leq n}$ is \emph{extracted} from the net $(F_i)_{i \in I}$ if there are indices
 $$
 i_1 < i_2 < \dots < i_n
 $$
 in $I$ such that $K_j = F_{i_j}$ for all $1 \leq j \leq n$.
 \par
  Let $\varepsilon > 0$ and let $n$ be a positive integer.
By   \eqref{e:boundary-Folner} and
\eqref{eq:lambda},  
it is clear that we can find, using induction on $n$, a finite sequence $(K_j)_{1 \leq j \leq n}$ extracted  from the  net $(F_i)_{i \in I}$ such that:
\begin{equation*}
 \alpha(K_k, K_j)\le \frac{\varepsilon^{2n}}{|K_j|} \quad \text{for all  } 1\le j < k \le n
\end{equation*}
and
 \begin{equation}
\label{eq:inequality xi}
\frac{h(K_j)}{|K_j|} \le \lambda + \varepsilon \quad \text{for all  } 1 \leq j \leq n.
 \end{equation}

 Suppose now that $0 < \varepsilon \leq  \dfrac{1}{2}$ and that $n \geq n_0$, where 
$n_0 = n_0(\varepsilon)$ is as in Theorem \ref{th:quasi-tile}.
 \par
Let $D \subset S$ be a non-empty finite subset satisfying
$\alpha(D,K_j)\le \varepsilon^{2n}$ for all $1\le j \le n$.
\par
By Theorem \ref{th:quasi-tile}, we can find a sequence $(P_j)_{1 \leq j \leq n}$ of finite subsets of $S$ satisfying the following conditions:  
\begin{enumerate}[\rm (T1)]
\item
 the set $P_j$ is an $(\varepsilon,K_j)$-filling pattern for $D$ for every $1 \leq j \leq n$;
\item
the subsets $K_jP_j \subset D$, $1 \leq j \leq n$,  are pairwise disjoint;
\item
the subset $D' \subset D$ defined by
$$
D' := D \setminus \bigcup_{1 \leq j \leq n} K_j P_j
$$
is such that $|D'| \leq \varepsilon |D|$.
\end{enumerate} 
 We then have
 $$
D = \bigcup_{1 \leq j \leq n} K_jP_j \cup D'. 
$$
 By applying  the subadditivity property (H1) of $h$, it follows that
\begin{equation}
\label{e:bound-hD}
h(D) \leq \sum_{1 \leq j \leq n} h(K_jP_j) + h(D').
\end{equation} 

 As $|D'| \le \varepsilon |D|$ by (T3),
we deduce from \eqref{e:h-bound} that
\begin{equation}
\label{e:bound-h-D-epsilon}
h(D') \leq M\varepsilon  |D|.
\end{equation}

On the other hand, for all $1 \leq j \leq n$, we have
\begin{align*}
h(K_jP_j) &= h\left(\bigcup_{s \in P_j} K_js\right) \\
&\leq \sum_{s \in P_j} h(K_j s) && \text{(by the subadditivity property (H1))} \\
&\leq \sum_{s \in P_j} h(K_j) && \text{(by the right-subinvariance property (H2))} \\
&= \sum_{s \in P_j} \frac{h(K_j)}{|K_j|} |K_j s| 
&& \text{(since $|K_j| = |K_j s|$ by right-cancellability of $s$)} \\
& \leq  (\lambda + \varepsilon)  \sum_{s \in P_j} |K_j s| && \text{(by \eqref{eq:inequality xi}).} 
\end{align*}
As the family $(K_j s)_{s \in P_j}$ is $\varepsilon$-disjoint by (T1), we then deduce from 
Lemma \ref{lemepsilondisjoint} that
$$
h(K_jP_j) \leq \frac{\lambda + \varepsilon}{1 - \varepsilon} \left| \bigcup_{s \in P_j} K_j s \right| = \frac{\lambda + \varepsilon}{1 - \varepsilon} |K_j P_j|.  
$$
This implies
$$
\sum_{1 \leq j \leq n} h(K_jP_j) \leq \frac{\lambda + \varepsilon}{1 - \varepsilon} \sum_{1 \leq j \leq n} |K_j P_j|
$$
and hence
\begin{equation}
\label{e:bound-h-union-KjPj}
\sum_{1 \leq j \leq n} h(K_jP_j) \leq \frac{\lambda + \varepsilon}{1 - \varepsilon} |D|,
\end{equation}
since the sets $K_j P_j$, $1 \leq j \leq n$, are pairwise disjoint subsets of $D$ by (T2).
\par
From \eqref{e:bound-hD}, \eqref{e:bound-h-D-epsilon}, and \eqref{e:bound-h-union-KjPj}, we deduce that
\begin{equation}
\label{eqnmajorationfinaleh(D)}
\frac{h(D)}{|D|} \leq \frac{\lambda + \varepsilon}{1 - \varepsilon} + M\varepsilon.
\end{equation}

By \eqref{e:boundary-Folner}, we can find $i_0 \in I$ such that,
for all $i \geq i_0$,
$$
\alpha(F_{i}, K_j) \le \varepsilon^{2n}\quad \text{for all  }1 \le j \le n.
$$
Hence, by replacing $D$ by $F_i$ for $i\geq i_0$ in inequality \eqref{eqnmajorationfinaleh(D)}, we obtain
\begin{equation*}\label{eqnmajorationfinaleh(D)Fi}
 \frac{h(F_i)}{|F_i|} \le \frac{\lambda + \varepsilon}{1-\varepsilon}+
 M\varepsilon . 
 \end{equation*} 
This implies 
 \begin{equation*} \nonumber
 \limsup_i\frac{h(F_i)}{|F_i|} \le
 \frac{\lambda+\varepsilon}{1-\varepsilon}+M\varepsilon .
 \end{equation*}
 Since the latter inequality is satisfied for all 
 $\varepsilon  \in (0,\dfrac{1}{2}]$, taking the limit as $\varepsilon$ tends to $0$, we obtain
 $$
 \limsup_{i} \frac{h(F_i)}{|F_i|} \le \lambda=
 \liminf_{i} \frac{h(F_i)}{|F_i|}.
 $$ 
This shows that \eqref{eq:lambda} is indeed a true limit.
\par
It only remains to show that $\lambda = \lim_i \dfrac{h(F_i)}{|F_i|}$ does not depend on the choice of the left-F\o lner net $(F_i)_{i \in I}$.
So suppose that $(G_j)_{j \in J}$ is another left-F\o lner net for $S$ and let 
$\nu = \lim_j \dfrac{h(G_j)}{|G_j|}$.
\par
 Take  disjoint copies $I'$ and $J'$ of the sets $I$ and $J$, i.e., sets $I'$ and $J'$ with $I \cap I' = \varnothing$ and $J \cap J' = \varnothing$ together with  bijective maps
 $\varphi \colon I \to I'$ and $\psi \colon J \to J'$.
 Consider the set $T = (I \times J) \cup (I' \times J')$
with the partial ordering  defined as follows.
Given $t_1,t_2 \in T$, we write $t_1 \leq t_2$ if and only if there exist indices $i_1,i_2 \in I$ and $j_1,j_2 \in J$ such that $i_1 \leq i_2$, $j_1 \leq j_2$, and
$$
(t_1 = (i_1,j_1) \text{  or  } t_1 = (\varphi(i_1),\psi(j_1))) \text{  and  } (t_2 = (i_2,j_2) \text{  or  } t_2 = (\varphi(i_2),\psi(j_2))). 
$$ 
Observe that  $(T,\leq)$ is a directed set since $(I,\leq)$ and $(J,\leq)$ are directed sets.
Now we define a net $(H_t)_{t \in T}$ of non-empty finite subsets of $S$ by setting
\begin{equation*}
H_t  =
\begin{cases}
F_i & \text{ if  } t = (i,j) \in I \times J, \\
G_j & \text{  if  } t = (\varphi(i),\psi(j)) \in I' \times J'. 
\end{cases}
\end{equation*}
 Clearly $(H_t)_{t \in T}$ is a left-F\o lner net for $S$.
By the first part of the proof, the net
$\left(\dfrac{h(H_t)}{|H_t|}\right)_{t \in T}$ converges to some $\tau \geq 0$.
Using the fact that  for every $t_1$ in $T$, there exits $t_2$  in $I \times J$ (resp. in $I' \times J'$) such that $t_1 \leq t_2$, we conclude that
$\tau = \lambda = \nu$.
This completes the proof of Theorem \ref{theOWGromov}.

% SECTION 5
\section{Applications to dynamical systems} 
\label{sec:applications}
 
\subsection*{Topological entropy}
(cf.  \cite{adler-entropy})
 Let $X$ be a compact topological space.
 \par
 An \emph{open cover} of $X$ is a family of open subsets of $X$ whose union is $X$.
  Let  $\UU = (U_j)_{j \in J}$ and $\VV = (V_k)_{k \in K}$ be two open covers of $X$.
 One says that $\VV$ is \emph{finer}  than  $\UU$,  and one writes $\VV \succ \UU$, 
 if, for each $k \in K$, there exists $j \in J$ such that 
$V_k \subset U_j$.
 One says that $\VV$ is a \emph{subcover}  of $\UU$ if $K \subset J$ 
 and $V_k = U_k$ for all $k \in K$.
One writes $\UU \cong \VV$ if  
$\{U_j : j \in J\} = \{V_k : k \in K\}$, that is, if   the open subsets of $X$  appearing in $\UU$ and $\VV$ are the same (as soon as we forget that they are indexed).
 \par
  The \emph{join}  of
 $\UU$ and $\VV$ 
is the open cover $\UU \vee \VV$ of $x$ defined by   
 $\UU \vee \VV := (U_j \cap V_k)_{(j,k) \in J \times K}$.
If $f \colon X \to X$ is a continuous map, the \emph{pullback} of $\UU$ by $f$ is the open cover  $f^{-1}(\UU)$ of $X$ defined by
$f^{-1}(\UU) := (f^{-1}(U_j))_{j \in J}$.
\par
Since $X$ is compact,   every open cover of $X$ admits a finite subcover.
Given an open cover $\UU $ of $X$,  let $N(\UU)$ denote  
 the smallest integer 
$n \geq 0$ such that $\UU$ admits a subcover of cardinality $n$.

 \begin{lemma}
\label{l:properties-N-open-cover}
Let $X$ be a compact space. 
Let   $\UU = (U_j)_{j \in J}$ and $\VV = (V_k)_{k \in K}$ be two open covers of $X$.
Then one has
\begin{enumerate}[\rm (i)]
% (i)
\item
  $N(\UU \vee \VV) \leq N(\UU)N(\VV)$;
% (ii)
\item
if  $\VV \succ \UU$ then $N(\VV) \geq N(\UU)$;
% (iii)
\item
if $\UU \cong \VV$ then $N(\UU) = N(\VV)$;
% (iv)
\item
if $f \colon X \to X$ is a continuous map then
 $N(f^{-1}(\UU)) \leq N(\UU)$.
\end{enumerate} 
\end{lemma}

\begin{proof}
These properties are all obvious (see for example \cite{adler-entropy}).
\end{proof}

Now suppose that the compact space $X$ is endowed with a continuous  action of a semigroup $S$.
This means that we are given a map $S \times X \to X$, $(s,x) \mapsto sx$, satisfying the following conditions:
(1) one has $s_1(s_2x) = (s_1s_2) x$ for all $s_1,s_2 \in S$ and $x \in X$;
(2) the map $T_s \colon X \to X$ defined by $T_s(x) := sx$ is continuous for all $s \in S$.
\par
Let $\UU$ be an open cover of $X$.
Consider the map $h_\UU \colon \PP_{fin}(S) \to \R$ defined by
\begin{equation}
\label{e:def-h-U}
h_\UU (A) := \log  N(\UU_A),
\end{equation}
where
\begin{equation}
\label{e:def-U-A}
\UU_A := \bigvee_{s \in A} T_s^{-1}(\UU).
\end{equation}
(By convention, $\UU_\varnothing = \{X\}$ so that $h_\UU(\varnothing) = 0$.)

\begin{proposition}
\label{p:properties-h-U}
Let $X$ be a compact space equipped with a continuous  action of a semigroup $S$ and let $\UU$ be an open cover of $X$.
Then the  map $h_\UU \colon \PP_{fin}(S) \to \R$ defined by \eqref{e:def-h-U} is non-decreasing, subadditive, 
right-subinvariant, and uniformly bounded on singletons.
 \end{proposition}
 
 \begin{proof}
Let $A$ and $B$ be finite subsets of $S$.
\par
If $A \subset B$, then $\UU_B$ is finer than 
 $\UU_A$. 
 This implies $N(\UU_A) \leq N(\UU_B)$ by Lemma \ref{l:properties-N-open-cover}.(ii)  and hence $h_\UU(A) \leq h_\UU(B)$.
 This shows that $h_\UU$ is non-decreasing.
 \par
 Suppose now that $A$ and $B$ are disjoint.
Then we have $\UU_{A \cup B} = \UU_A \vee \UU_B$  and hence
$N(\UU_{A \cup B}) \leq N(\UU_A)N(\UU_B)$. This implies
$h_\UU(A \cup B) \leq h_\UU(A) + h_\UU(B)$.
 \par
If $A$ and $B$ are arbitrary subsets of $S$,
 we can write
 \begin{align*}
h_\UU(A \cup B) &= h_\UU((A \setminus B) \cup B) \\
&\leq h_\UU(A \setminus B) + h_\UU(B) && \text{(since $A \setminus B$ and $B$ are disjoint)} \\
& \leq h_\UU(A) + h_\UU(B) && \text{(since $h$ is non-decreasing).}
\end{align*}
this shows that $h_\UU$ is subadditive.
 \par
To prove right-subinvariance,  we first observe that, for every $s \in S$ and any finite subset $A$ of $S$, we have 
\begin{align*}
\UU_{As} & = \bigvee_{t \in As} T_t^{-1}(\UU) \\
& \cong \bigvee_{a \in A} T_{as}^{-1}(\UU) \\
& = \bigvee_{a \in A}(T_a \circ T_s)^{-1}(\UU) \\ 
& = \bigvee_{a \in A} T_s^{-1}(T_a^{-1}(\UU)) \\
& = T_s^{-1}\left(\bigvee_{a \in A} T_a^{-1}(\UU)\right) \\
& = T_s^{-1}(\UU_A).
\end{align*}
 We then deduce that
$$
h_\UU(As) = \log N(\UU_{As}) = \log N(T_s^{-1}(\UU_A)) \leq \log N(\UU_A) = h_\UU(A)
$$
by using assertions (iii) and (iv) in Lemma \ref{l:properties-N-open-cover}. This shows that 
$h_\UU$ is right-subinvariant.
\par
  Finally, for all $s \in S$, we have
 $$
 h_\UU(\{s\}) = \log N(T_s^{-1}(\UU)) \leq \log N(\UU)
 $$
 by Lemma \ref{l:properties-N-open-cover}.(iv). It follows that $h_\UU$ is uniformly bounded on singletons.
 \end{proof}

From Proposition \ref{p:properties-h-U} and Theorem \ref{theOWGromov}, we deduce the following result.

\begin{theorem}
Let $X$ be a compact space equipped with a continuous action of a cancellative left-amenable semigroup $S$ and let $\UU$ be an open cover of $X$.
Then, for every left-F\o lner net $(F_i)_{i \in I}$ of $S$, the limit
$$
\eta_\UU := \lim_i \frac{h_\UU(F_i)}{|F_i|}
$$
exists and is finite. Moreover, $\eta_\UU$ does not depend on the choice of the left-F\o lner net $(F_i)_{i \in I}$.
\end{theorem}

The quantity  $0 \leq \eta \leq + \infty$ defined by $\eta := \sup_\UU \eta_\UU$, where $\UU$ runs over all open covers of $X$, is the \emph{topological entropy} of the continuous dynamical system $(X,S)$. 

\subsection*{Topological mean dimension}
(cf. \cite{gromov}, \cite{lindenstrauss-weiss}, \cite{coornaert-krieger}, \cite{coornaert-smf})
Let $X$ be a compact metrizable space.
\par
Let $\UU = (U_j)_{j \in J}$ be a finite open cover of $X$.
The \emph{local order} of $\UU$ at a point $x \in X$ is the integer
$\ord(\UU,x) := 1 + m(\UU,x)$, where $m(\UU,x)$ is the number of indices $j \in J$ such that $x \in U_j$.
The \emph{order} of $\UU$ is the integer $\ord(\UU) := \max_{x \in X} \ord(\UU,x)$.
Define the integer $D(\UU)$ by 
$D(\UU) := \min_\VV \ord(\VV)$,
where $\VV$ runs over all finite open covers of $X$ such that $\VV \succ \UU$.
The quantity $0 \leq \dim(X) \leq +\infty$ defined by $\dim(X) := \sup_\UU D(\UU)$, where $\UU$ runs over all finite open covers of $X$, is the \emph{topological dimension} of $X$
(cf. \cite{hurewicz-wallman}).
 
 \begin{lemma}
\label{l:properties-D-open-cover}
Let $X$ be a compact metrizable space. 
Let   $\UU = (U_j)_{j \in J}$ and $\VV = (V_k)_{k \in K}$ be two finite  open covers of $X$.
Then one has
\begin{enumerate}[\rm (i)]
% (i)
\item
  $D(\UU \vee \VV) \leq D(\UU) + D(\VV)$;
% (ii)
\item
if  $\VV \succ \UU$ then $D(\VV) \geq D(\UU)$;
% (iii)
\item
if $\UU \cong \VV$ then $D(\UU) = D(\VV)$;
% (iv)
\item
if $f \colon X \to X$ is a continuous map then
 $D(f^{-1}(\UU)) \leq D(\UU)$.
\end{enumerate} 
\end{lemma}

\begin{proof}
See for example \cite{lindenstrauss-weiss}, \cite{coornaert-krieger}, or \cite{coornaert-smf}.
\end{proof}

Let $X$ be a compact metrizable space equipped with a continuous action of a semigroup $S$. 
Let $\UU$ be a finite open cover of $X$.
Consider the map $h_\UU^{\text{dim}}  \colon \PP_{fin}(S) \to \R$ defined by
\begin{equation}
\label{e:def-hdim-U}
h_\UU^{\text{dim}} (A) :=   D(\UU_A),
\end{equation}
where $\UU_A$ is defined by \eqref{e:def-U-A}.
 
\begin{proposition}
\label{p:properties-hdim-U}
Let $X$ be a compact metrizable space equipped with a continuous  action of a semigroup $S$ and let $\UU$ be a finite  open cover of $X$.
Then the  map $h_\UU^{\text{dim}} \colon \PP_{fin}(S) \to \R$ defined by \eqref{e:def-hdim-U} is non-decreasing, subadditive, 
right-subinvariant, and uniformly bounded on singletons.
 \end{proposition}

\begin{proof}
Mutatis mutandis, the proof  is  that of Proposition \ref{p:properties-h-U} with 
Lemma \ref{l:properties-D-open-cover} replacing Lemma \ref{l:properties-N-open-cover}.
\end{proof} 

From Proposition \ref{p:properties-hdim-U} and Theorem \ref{theOWGromov}, we deduce the following result.

\begin{theorem}
Let $X$ be a compact metrizable space equipped with a continuous action of a cancellative 
left-amenable semigroup $S$ and let $\UU$ be a finite  open cover of $X$.
Then, for every left-F\o lner net $(F_i)_{i \in I}$ of $S$, the limit
$$
\eta_\UU^{\text{dim}}  := \lim_i \frac{h_\UU^{\text{dim}} (F_i)}{|F_i|}
$$
exists and is finite. Moreover, $\eta_\UU^{\text{dim}}$ does not depend on the choice of the left-F\o lner net $(F_i)_{i \in I}$.
\end{theorem}

The quantity  $0 \leq \eta^{\text{dim}}  \leq + \infty$ defined by $\eta^{\text{dim}}  := \sup_\UU \eta_\UU^{\text{dim}}$, where $\UU$ runs over all finite open covers of $X$, is the \emph{topological mean dimension} of the continuous dynamical 
system $(X,S)$.

\subsection*{Measure-theoretic entropy}
(cf. \cite{kolmogorov-entropy}, \cite{sinai}, \cite{katok-hasselblatt})
Let $X = (X,\BB,p)$ be a probability space.
\par
A \emph{finite measurable partition} of $X$ is a finite family $\UU = (U_j)_{j \in J}$ of pairwise disjoint measurable subsets of $X$ whose union is $X$ (here, equalities for subsets of $X$ are understood to hold  up to null-measure sets).
The join operation $\vee$, as well as  the relations $\succ$ and $\cong$, can also  be defined for finite measurable partitions.
Moreover, if $T \colon X \to X$ is a measurable map and $\UU = (U_j)_{j \in J}$ is a finite measurable partition of $X$, then $T^{-1}(\UU) := (T^{-1}(U_j))_{j \in J}$ is also a finite measurable partition of $X$.
\par
If $\UU = (U_j)_{j \in J}$ is a finite measurable partition of $X$, we define the real number
$E(\UU) \geq 0$ by 
$$
E(\UU) := - \sum_{j \in J} p(U_j) \log p(U_j),
$$
with the usual convention $0 \log 0 = 0$.
\par
  A measurable map $T \colon X \to X$ is said to be  \emph{measure-preserving} if $p(T^{-1}(B)) = p(B)$ for all $B \in \BB$.

 \begin{lemma}
\label{l:properties-K-measurable-partition}
Let $(X,\BB,p)$ be a probability space. 
Let   $\UU = (U_j)_{j \in J}$ and $\VV = (V_k)_{k \in K}$ be two finite  measurable partitions of $X$.
Then one has
\begin{enumerate}[\rm (i)]
% (i)
\item
  $E(\UU \vee \VV) \leq E(\UU) + E(\VV)$;
% (ii)
\item
if  $\VV \succ \UU$ then $E(\VV) \geq E(\UU)$;
% (iii)
\item
if $\UU \cong \VV$ then $E(\UU) = E(\VV)$;
% (iv)
\item
if $T \colon X \to X$ is a measure-preserving map then
 $E(T^{-1}(\UU)) = E(\UU)$.
\end{enumerate} 
\end{lemma}

\begin{proof}
See for example \cite[Section 4.3]{katok-hasselblatt}.
\end{proof}

Let $(X,\BB,p)$ be a probability space.
Suppose that $X$ is equipped with a \emph{measure-preserving action} of a semigroup $S$, that is, a family of measure-preserving  maps
$T_s \colon X \to X$, $s \in S$, such that
$$
T_{s_1} \circ T_{s_2} = T_{s_1 s_2} \quad  p-\text{a.e.} 
$$
for all $s_1,s_2 \in S$.
\par
 Let $\UU$ be a finite measurable partition of $X$.
Consider the map $h_\UU^{\text{KS}}  \colon \PP_{fin}(S) \to \R$ defined by
\begin{equation}
\label{e:def-hdim-U-KS}
h_\UU^{\text{KS}} (A) :=   E(\UU_A),
\end{equation}
where $\UU_A$ is defined by \eqref{e:def-U-A}.
 
\begin{proposition}
\label{p:properties-hK-U}
Let $(X,\BB,p)$ be a probability space equipped with a measure-preserving  action of a semigroup $S$ and let $\UU$ be a finite  measurable partition of $X$.
Then the  map $h_\UU^{\text{KS}} \colon \PP_{fin}(S) \to \R$ defined by \eqref{e:def-hdim-U-KS} is non-decreasing, subadditive, 
right-invariant, and uniformly bounded on singletons.
 \end{proposition}

\begin{proof}
Mutatis mutandis, the proof  is  that of Proposition \ref{p:properties-h-U} with 
Lemma \ref{l:properties-K-measurable-partition} replacing Lemma \ref{l:properties-N-open-cover}.
Note that $h_\UU^{\text{KS}}$ is indeed right-invariant since in Lemma \ref{l:properties-K-measurable-partition}.(iv) an equality holds.
\end{proof} 

From Proposition \ref{p:properties-hK-U} and Theorem \ref{theOWGromov}, we deduce the following result.

\begin{theorem}
Let $(X,\BB,p)$ be a probability space equipped with a measure-preserving action of a cancellative 
left-amenable semigroup $S$ and let $\UU$ be a finite  measurable partition of $X$.
Then, for every left-F\o lner net $(F_i)_{i \in I}$ of $S$, the limit
$$
\eta_\UU^{\text{KS}}  := \lim_i \frac{h_\UU^{\text{KS}} (F_i)}{|F_i|}
$$
exists and is finite. Moreover, $\eta_\UU^{\text{KS}}$ does not depend on the choice of the left-F\o lner net $(F_i)_{i \in I}$.
\end{theorem}

The quantity  $0 \leq \eta^{\text{KS}}  \leq + \infty$ defined by $\eta^{\text{KS}}  := \sup_\UU \eta_\UU^{\text{KS}}$, where $\UU$ runs over all finite measurable partitions of $X$, is 
the \emph{measure-theoretic entropy}, or \emph{Kolmogoroff-Sinai entropy}, of the 
measure-preserving dynamical  system $(X,S)$.

         %%% REFERENCES


\begin{thebibliography}{10}

\bibitem{adler-entropy}
{\sc R.~L.~Adler, A.~G.~Konheim, and M.~H.~McAndrew}, {\em Topological
  entropy}, Trans. Amer. Math. Soc., 114 (1965), pp.~309--319.

\bibitem{coornaert-smf}
{\sc M.~Coornaert}, {\em Dimension topologique et syst\`emes dynamiques},
  vol.~14 of Cours Sp\'ecialis\'es [Specialized Courses], Soci\'et\'e
  Math\'ematique de France, Paris, 2005.

\bibitem{coornaert-krieger}
{\sc M.~Coornaert and F.~Krieger}, {\em Mean topological dimension for actions
  of discrete amenable groups}, Discrete Contin. Dyn. Syst., 13 (2005),
  pp.~779--793.

\bibitem{day-trans}
{\sc M.~M. Day}, {\em Means for the bounded functions and ergodicity of the
  bounded representations of semi-groups}, Trans. Amer. Math. Soc., 69 (1950),
  pp.~276--291.

\bibitem{day-illinois}
\leavevmode\vrule height 2pt depth -1.6pt width 23pt, {\em Amenable
  semigroups}, Illinois J. Math., 1 (1957), pp.~509--544.

\bibitem{day-survey}
\leavevmode\vrule height 2pt depth -1.6pt width 23pt, {\em Semigroups and
  amenability}, in Semigroups ({P}roc. {S}ympos., {W}ayne {S}tate {U}niv.,
  {D}etroit, {M}ich., 1968), Academic Press, New York, 1969, pp.~5--53.

\bibitem{fekete}
{\sc M.~Fekete}, {\em \"{U}ber die {V}erteilung der {W}urzeln bei gewissen
  algebraischen {G}leichungen mit ganzzahligen {K}oeffizienten}, Math. Z., 17
  (1923), pp.~228--249.

\bibitem{folner}
{\sc E.~F{\o}lner}, {\em On groups with full {B}anach mean value}, Math.
  Scand., 3 (1955), pp.~243--254.

\bibitem{frey}
{\sc A.~H. Frey, Jr}, {\em Studies on amenable semigroups}, ProQuest LLC, Ann
  Arbor, MI, 1960.
\newblock Thesis (Ph.D.)--University of Washington.

\bibitem{gromov}
{\sc M.~Gromov}, {\em Topological invariants of dynamical systems and spaces of
  holomorphic maps. {I}}, Math. Phys. Anal. Geom., 2 (1999), pp.~323--415.

\bibitem{hurewicz-wallman}
{\sc W.~Hurewicz and H.~Wallman}, {\em Dimension {T}heory}, Princeton
  Mathematical Series, v. 4, Princeton University Press, Princeton, N. J.,
  1941.

\bibitem{katok-hasselblatt}
{\sc A.~Katok and B.~Hasselblatt}, {\em Introduction to the modern theory of
  dynamical systems}, vol.~54 of Encyclopedia of Mathematics and its
  Applications, Cambridge University Press, Cambridge, 1995.
\newblock With a supplementary chapter by Katok and Leonardo Mendoza.

\bibitem{kolmogorov-entropy}
{\sc A.~N.~Kolmogorov}, {\em A new metric invariant of transient dynamical
  systems and automorphisms in {L}ebesgue spaces}, Dokl. Akad. Nauk SSSR
  (N.S.), 119 (1958), pp.~861--864.

\bibitem{krieger}
{\sc F.~Krieger}, {\em Le lemme d'{O}rnstein-{W}eiss d'apr\`es {G}romov}, in
  Dynamics, ergodic theory, and geometry, vol.~54 of Math. Sci. Res. Inst.
  Publ., Cambridge Univ. Press, Cambridge, 2007, pp.~99--111.

\bibitem{lindenstrauss-weiss}
{\sc E.~Lindenstrauss and B.~Weiss}, {\em Mean topological dimension}, Israel
  J. Math., 115 (2000), pp.~1--24.

\bibitem{namioka}
{\sc I.~Namioka}, {\em F\o lner's conditions for amenable semi-groups}, Math.
  Scand., 15 (1964), pp.~18--28.
  
\bibitem{vN}
{\sc J.~von Neumann},
{\em Zur Allgemeine Theorie des Masses}, 
Fund. Math., 13 (1929), pp.~73--116.

\bibitem{ornstein-weiss}
{\sc D.~S.~Ornstein and B.~Weiss}, {\em Entropy and isomorphism theorems for
  actions of amenable groups}, J. Analyse Math., 48 (1987), pp.~1--141.

\bibitem{sinai}
{\sc Ya.~G. Sina{\u\i}}, {\em On the concept of entropy for a dynamic system},
  Dokl. Akad. Nauk SSSR, 124 (1959), pp.~768--771.

\end{thebibliography}
\end{document}